\documentclass{amsart}

\usepackage{amsmath,amssymb,amsthm,amscd}

\newtheorem{prop}{Proposition}[section]

\newtheorem{lem}[prop]{Lemma}

\theoremstyle{definition}

\newtheorem*{claim}{Claim}

\newtheorem{rem}[prop]{Remark}

\newtheorem*{ack}{Acknowledgement}

\def\co{\colon\thinspace}

\newcommand{\C}{\mathbb{C}}
\newcommand{\CP}{\mathbb{C}\mathrm{P}}

\newcommand{\rme}{\mathrm{e}}

\newcommand{\Hp}{\mathbb H}

\newcommand{\rmi}{\mathrm{i}}

\newcommand{\rmj}{\mathrm{j}}

\newcommand{\rmk}{\mathrm{k}}

\newcommand{\N}{\mathbb{N}}

\newcommand{\R}{\mathbb{R}}
\newcommand{\RP}{\mathbb{R}\mathrm{P}}

\newcommand{\oz}{\overline{z}}
\newcommand{\Z}{\mathbb{Z}}

\DeclareMathOperator{\piorb}{\pi_1^{\mathrm{orb}}}

\begin{document}

\author[H.~Geiges]{Hansj\"org Geiges}

\author[C.~Lange]{Christian Lange}
\address{Mathematisches Institut, Universit\"at zu K\"oln,
Weyertal 86--90, 50931 K\"oln, Germany}
\email{geiges@math.uni-koeln.de, clange@math.uni-koeln.de}

\title[Seifert fibrations of lens spaces]{Seifert fibrations of lens
spaces over non-orientable bases}

\date{}

\begin{abstract}
We classify the Seifert fibrations of lens spaces where the base
orbifold is non-orientable. This is an addendum to our
earlier paper `Seifert fibrations of lens spaces'. We
correct Lemma~4.1 of that paper and fill the gap in the
classification that resulted from the erroneous lemma.
\end{abstract}



\thanks{The authors are partially supported by the SFB/TRR 191
`Symplectic Structures in Geometry, Algebra and Dynamics',
funded by the DFG (Projektnummer 281071066 -- TRR 191).}

\maketitle


\section{Introduction}
The aim of this note is to determine the $3$-dimensional lens spaces that
admit a Seifert fibration over a non-orientable base orbifold, and to
classify these fibrations.

In our earlier paper \cite{gela18} we presented an algorithm for finding
all Seifert fibrations of any given lens space.
When the base orbifold is non-orientable, we claimed in
\cite[Lemma~4.1]{gela18} that the base is topologically $\RP^2$
(which is correct), and that there are no singular fibres
(which turns out to be false). As was kindly pointed out to us by
Tye Lidman (jointly with Liam Watson and Robert DeYeso),
the lens spaces $L(4n,2n\pm 1)$
admit, for $n\geq 2$, a Seifert fibration with one singular fibre
over $\RP^2$. As we shall see in Proposition~\ref{prop:main},
these are the only Seifert fibrations over non-orientable bases
missing from the classification in~\cite{gela18}.

The results of \cite{gela18} concerning Seifert fibrations of lens
spaces over orientable bases are not affected by the gap in
\cite[Lemma~4.1]{gela18}.

We follow the notations and conventions of~\cite{gela18} as regards Seifert
invariants. We assume that the reader has a copy of that paper at hand.
The only new notation is $\Sigma(n_1,\ldots, n_k)$
for a $2$-dimensional orbifold whose underlying topological surface
is $\Sigma$, with $k$ orbifold points of order $n_1,\ldots, n_k$,
respectively.
\section{The base of the fibration}
Here we correct Lemma 4.1 from~\cite{gela18}.

\begin{lem}
If a lens space admits a Seifert fibration over a non-orientable base,
the base is $\RP^2$ (as a topological surface), and there is at most
one singular fibre. In other words, the base orbifold is
$\RP^2(n)$ for some $n\in\N=\{1,2,\ldots\}$.
\end{lem}

\begin{proof}
The argument in \cite{gela18} for showing that the base has finite cyclic
fundamental group and hence is topologically $\RP^2$ is correct.
Also, the \emph{orbifold} fundamental group of the base orbifold is
a quotient of the fundamental group of the total space
\cite[Section~2.3]{gela18}, and hence must likewise be finite cyclic
in our situation.

The orbifold fundamental group of $\RP^2(n_1,\ldots, n_k)$ admits
the presentation
\[ \piorb\bigl(\RP^2(n_1,\ldots, n_k)\bigr)\cong
\langle a,q_1,\ldots, q_k\,|\, q_1^{n_1},\ldots,q_k^{n_k},\,
q_1\cdots q_ka^2 \rangle. \]
The error in \cite{gela18} occurred in our claim that this group
is abelian if and only if $n_i=1$ for all $i=1,\ldots, k$.
This is plainly wrong, as shown by the isomorphism
\[ \piorb\bigl(\RP^2(n)\bigr)\cong
\langle a,q\,|\,q^n,\, qa^2\rangle\cong
\langle a\,|\, a^{2n}\rangle\cong\Z_{2n}.\]

To prove the lemma, it suffices to establish the
following claim.

\begin{claim}
The orbifold fundamental group of $\RP^2(n_1,\ldots, n_k)$
is finite if and only if at most one $n_i$ is greater than~$1$.
\end{claim}

The `if' direction is given by the example above. For the `only if' direction
we observe that $\piorb\bigl(\RP^2(n_1,n_2)\bigr)$ is a quotient
group of $\piorb\bigl(\RP^2(n_1,n_2,\ldots, n_k)\bigr)$,
obtained by dividing out the normal subgroup generated by $q_3,\ldots, q_k$.
Therefore, it suffices to show that $\piorb\bigl(\RP^2(n_1,n_2)\bigr)$
is not finite for $n_1,n_2>1$.

One way to argue is as follows. By \cite[Theorem~2.3]{scot83},
the orbifold $\RP^2(n_1,n_2)$ is good, i.e.\ it is covered by a
manifold, and then by \cite[Theorem~2.5]{scot83} it is in fact
\emph{finitely} covered by a compact smooth surface.
The Euler characteristic of
$\RP^2(n_1,n_2)$ equals, by \cite[page~427]{scot83},
\[ \chi\bigl(\RP^2(n_1,n_2)\bigr)=1-\Bigl(1-\frac{1}{n_1}\Bigr)
-\Bigl(1-\frac{1}{n_2}\Bigr)=-1+\frac{1}{n_1}+\frac{1}{n_2},\]
which is non-positive for $n_1,n_2>1$. By its definition,
the Euler characteristic is multiplicative under finite coverings,
so $\RP^2(n_1,n_2)$ is covered by a compact smooth surface of
non-positive Euler characteristic, hence by~$\R^2$, and so
the orbifold fundamental group is infinite.

The classification of good vs.\ bad orbifolds as described in
\cite{scot83} rests on the existence of universal covering orbifolds
and the construction of proper coverings for the orbifolds not
on the list of bad ones.
For $\RP^2(n_1,n_2)$, $n_1,n_2>1$, we can explicitly exhibit
an infinite tower of proper coverings to give a direct proof
that $\piorb\bigl(\RP^2(n_1,n_2)\bigr)$ is not finite.
Consider $S^2(n_1,n_1,n_2,n_2)$, with the orbifold points of
order $n_i$ arranged in antipodal position on the round $2$-sphere,
$i=1,2$. The quotient under the antipodal map then gives a $2$-fold covering
\[ S^2(n_1,n_1,n_2,n_2)\longrightarrow \RP^2(n_1,n_2).\]
Next, we consider $S^2(n_2,\ldots, n_2)$ with $2n_1\geq 4$ orbifold
points of order $n_2$ arranged symmetrically along an equator. The quotient
under the rotation about the poles through an angle $2\pi/n_1$ defines
an $n_1$-fold covering
\[ S^2(\underbrace{n_2,\ldots,n_2}_{2n_1})\longrightarrow
S^2(n_1,n_1,n_2,n_2).\]
In the same fashion, we construct an $n_2$-fold covering
\[ S^2(\underbrace{n_2,\ldots,n_2}_{n_2(2n_1-2)})\longrightarrow
S^2(\underbrace{n_2,\ldots,n_2}_{2n_1}).\]
Since $n_2(2n_1-2)\geq 4$, we can continue \emph{ad infinitum}.
This again proves the claim (and the lemma).
\end{proof}

Thus, we are left with classifying the Seifert fibrations with
base $\RP^2(n)$, $n\in\N$, and total space equal to some lens space.
\section{Lens spaces fibring over $\RP^2(n)$}
The case $n=1$ is dealt with in \cite[Proposition~4.2]{gela18}:
the total space can be $L(4,1)$ or $L(4,3)$, and either of these admits
a unique Seifert fibration over $\RP^2$.

Let $M$ be a lens space that admits a Seifert fibration over $\RP^2(n)$,
$n\geq 2$. Using the equivalences of Seifert invariants described in
\cite[Section~2.2]{gela18}, we may absorb any non-singular fibre of type
$(1,b)$, contributing to the Euler class of the bundle, into the invariants
of the singular fibre, and hence write the Seifert fibration of
$M$ in the notation of \cite{gela18} as
\[ M=M\bigl(-1;(n,\beta)\bigr).\]
Then the fundamental group of $M$ has the presentation,
by \cite[Section~2.3]{gela18},
\[ \pi_1(M)\cong\langle a,q,h\,|\, a^{-1}ha=h^{-1},\,
[h,q],\, q^nh^{\beta},\, qa^2\rangle.\]
The inverse of the first relation is $a^{-1}h^{-1}a=h$, and hence
we have $a^{-2}ha^2=h$. It follows that the presentation is
equivalent to
\[ \pi_1(M)\cong\langle a,h\,|\, a^{-1}ha=h^{-1},\,
a^{2n}=h^{\beta}\rangle.\]
The first relation implies $a^{-1}h^{\beta}a=h^{-\beta}$; the second,
$a^{-1}h^{\beta}a=h^{\beta}$. With the resulting (superfluous)
relation $h^{2\beta}=1$, we recognise the presentation of $\pi_1(M)$
by a result of H\"older, cf.~\cite[Lemma~2.1]{hemp00}, as that of
a metacyclic group
\[ \Z_{2\beta}\rightarrowtail \pi_1(M)\twoheadrightarrow\Z_{2n}.\]
Indeed, we may assume $\beta>0$ in the presentation of $\pi_1(M)$,
possibly after replacing the generator $h$ by $h^{-1}$, if necessary.
Then, in the notation of \cite[Lemma~2.1]{hemp00}, one has to set
$x=a$, $y=h$, as well as $k=2n$, $l=\beta$, $m=2\beta$, and $n=2\beta-1$;
then the divisibility conditions $m|(n^k-1)$ and $m|l(n-1)$ are satisfied.
The normal subgroup of order $2\beta$ is generated by~$h$. Every element
of $\pi_1(M)$ can be written uniquely in the form $a^kh^{\ell}$
with $0\leq k\leq 2n-1$ and $0\leq\ell\leq 2\beta-1$, and the projection
to $\Z_{2n}$ is given by $a^kh^{\ell}\mapsto k$.

Here is the main result of this note.

\begin{prop}
\label{prop:main}
The only lens spaces that admit a Seifert fibration over $\RP^2(n)$,
$n\geq 1$, are $L(4n,2n\pm 1)$, and these Seifert fibrations are unique.
\end{prop}

\begin{proof}
We are going to show that the Seifert manifold $M\bigl(-1;(n,\beta)\bigr)$
is a lens space if and only if $\beta=\pm 1$, and that there is 
an orientation-preserving diffeomorphism
\[ M\bigl(-1;(n,\pm 1)\bigr)\cong L(4n,2n\mp 1).\]

The `only if' part follows from the presentation of $\pi_1(M)$, which
is that of a cyclic group only if $a$ commutes with~$h$, which forces
$h^2=1$ and hence $\beta=\pm 1$.

For the `if' part it suffices to exhibit an explicit Seifert fibration
of $L(4n,2n\mp 1)$ over $\RP^2(n)$. Recall that the positive Hopf
fibration of the $3$-sphere is the map
\[ \begin{array}{rcl}
\C^2\supset S^3 & \longrightarrow & S^2=\CP^1\\
(z_1,z_2)       & \longmapsto     & [z_1:z_2].
\end{array}\]
The map
\[ A_+\co (z_1,z_2)\longmapsto
(\rme^{\pi\rmi/2n}\oz_2,\rme^{-\pi\rmi/2n}\oz_1)\]
defines an element of $\mathrm{SO}(4)$, and it sends
the Hopf fibres, i.e.\ the orbits of the $S^1$-action
\[ \theta(z_1,z_2)=(\rme^{\rmi\theta}z_1,\rme^{\rmi\theta}z_2),\;\;\;
\theta\in\R/2\pi\Z,\]
to Hopf fibres, reversing their orientation.
The square of $A_+$ is
\[ A_+^2\co (z_1,z_2)\longrightarrow
(\rme^{\pi\rmi/n}z_1,\rme^{-\pi\rmi/n}z_2),\]
which obviously defines a free $\Z_{2n}$-action on~$S^3$.
This induces a $\Z_n$-action (\emph{sic}!) on the quotient
$S^3/S^1=\CP^1$ with two fixed points at $[1:0]$ and $[0:1]$ and
quotient $S^2(n,n)$, so we have a Seifert fibration
\[ S^3/\Z_{2n}\longrightarrow S^2(n,n).\]
In turn, $A_+$ induces a $\Z_2$-action on $S^2(n,n)$.
This action is free. Indeed, if the class of
$(z_1^0,z_2^0)$ in $S^2(n,n)=(S^3/\Z_{2n})/S^1=(S^3/S^1)/\Z_n$
were fixed by~$A_+$,
there would be an odd integer $k$ and an angle $\theta\in\R/2\pi\Z$ such that
$A_+^k(z_1^0,z_2^0)=\theta(z_1^0,z_2^0)$. But
\[ A_+^k(z_1^0,z_2^0)=(\rme^{k\pi\rmi/2n}\oz_2^0,\rme^{-k\pi\rmi/2n}\oz_1^0)
\;\;\;\text{for $k$ odd},\]
so the fixed point condition would imply
\[ \rme^{k\pi\rmi/2n}\oz_2^0=\rme^{\rmi\theta}z_1^0\;\;\text{and}\;\;
\rme^{-k\pi\rmi/2n}\oz_1^0=\rme^{\rmi\theta}z_2^0\;\;\text{for some}\;\;
\theta\in\R/2\pi\Z.\]
Substituting these two equations into each other, we find
\[ z_1^0=\rme^{k\pi\rmi/n}z_1^0\;\;\text{and}\;\;
z_2^0=\rme^{-k\pi\rmi/n}z_2^0,\]
which would yield the contradiction that $2n$ divides~$k$.
In particular, $A_+$ generates a free $\Z_{4n}$-action on~$S^3$, and
we have a Seifert fibration
\[ S^3/\Z_{4n}\longrightarrow\RP^2(n).\]

Next we consider the free $\Z_{4n}$-action on $S^3$ generated by
\[ A_{2n-1}\co (z_1,z_2)\longmapsto (\rme^{\pi\rmi/2n}z_1,
\rme^{\pi\rmi(2n-1)/2n}z_2).\]
This yields the quotient $L(4n,2n-1)$. We now want to show
that $A_+$ and $A_{2n-1}$ are conjugate in $\mathrm{SO}(4)$,
which proves the existence of a Seifert fibration
$L(4n,2n-1)\rightarrow\RP(n)$.

In quaternionic notation $z_1+z_2\rmj=:a_0+a_1\rmi+a_2\rmj+a_3\rmk=:a\in
S^3\subset\Hp$, the maps $A_+$ and $A_{2n-1}$ take the form
\[ A_+(a)=-\sin(\pi/2n)\,\rmk a+\cos(\pi/2n)\,\rmk a\rmi\]
and
\[ A_{2n-1}(a)=\sin(\pi/2n)\,\rmi a-\cos(\pi/2n)\,\rmi a\rmi.\]
With $\Phi\in\mathrm{SO}(4)$ defined by
\[ \Phi(a)=\frac{1+\rmi-\rmj-\rmk}{2}\, a\]
we have $\Phi\circ A_+=A_{2n-1}\circ\Phi$.

In order to find a Seifert fibration $L(4n,2n+1)\rightarrow\RP^2(n)$,
consider the conjugation
$A_-$ of $A_+$ by the orientation-reversing diffeomorphism
$(z_1,z_2)\mapsto (z_1,\oz_2)$ of~$S^3$. Then argue analogously,
or observe that the quotient of
$S^3$ under this conjugated $\Z_{4n}$-action generated
by $A_-$ is $L(4n,2n-1)$
with reversed orientation, which is $L(4n,2n+1)$,
cf.\ \cite[Section~3.2]{gela18}.

Since the $M\bigl(-1;(n,\pm 1)\bigr)$ are the only
Seifert fibrations over $\RP^2(n)$ with total space
having a finite cyclic fundamental group, each of the lens spaces
$L(4n,2n\pm 1)$ has a unique Seifert fibration over $\RP^2(n)$.

The correct identification of the signs in the diffeomorphism
$M\bigl(-1;(n,\pm1)\bigr)\cong L(4n,2n\mp 1)$ (as oriented manifolds)
follows from \cite[Theorem~4.4]{gela18} and \cite[Theorem~5.1]{jane83}.
\end{proof}

\begin{rem}
(1) As in \cite{gela18} one can find a conjugating map $\Phi$
that conjugates $A_+$ and $A_-$ simultaneously to the
standard action giving the quotient $L(4n,2n\mp 1)$.

(2) Without explicit computation, one knows from
\cite{scot83} that $A_+$ must be conjugate in $\mathrm{SO}(4)$
to a map of the form
\[ A_{\tilde{q}}\co (z_1,z_2)\longmapsto
(\rme^{\pi\rmi/2n}z_1,\rme^{\pi\rmi\tilde{q}/2n}z_2)\]
for some $\tilde{q}\in\{1,\ldots,4n-1\}$ coprime to $4n$. Since
$\mathrm{trace}(A_{\tilde{q}})=\mathrm{trace}(A_+)=0$, we must have
$\cos(\pi\tilde{q}/2n)=-\cos(\pi/2n)$, that is, $\tilde{q}=2n\pm 1$.
This gives a Seifert fibration of one of $L(4n,2n\pm 1)$ over
$\RP^2(n)$, and one then concludes as before.

(3) A lens space that Seifert fibres over $\RP^2(n)$ can be
decomposed into a Seifert fibration over $D^2(n)$ and an
$S^1$-bundle over the M\"obius band. Since the lens space is
orientable, the latter bundle restricts over the soul of the
M\"obius band to the $S^1$-fibration of the Klein
bottle over~$S^1$. As shown by Bredon--Wood~\cite{brwo69}, the only
lens spaces that contain embedded Klein bottles are the $L(4n,2n\pm 1)$.
For an explicit description
of a Klein bottle in $L(4,1)$ see~\cite{lrs15}.
\end{rem}
\begin{ack}
We are most grateful to Tye Lidman, Liam Watson and Robert DeYeso
for pointing out the gap in \cite[Lemma~4.1]{gela18}
and for the reference to the work of Bredon--Wood and
Levine--Ruberman--Strle.
\end{ack}


\begin{thebibliography}{10}
%
\bibitem{brwo69}
Bredon, G. E., Wood, J. W.:
Non-orientable surfaces in orientable $3$-manifolds.
Invent. Math.
\textbf{7}, 83--110 (1969).
%
\bibitem{gela18}
Geiges, H., Lange, C.:
Seifert fibrations of lens spaces.
Abh. Math. Semin. Univ. Hambg.
\textbf{88}, 1--22 (2018)
%
\bibitem{hemp00}
Hempel, C. E.:
Metacyclic groups.
Comm. Algebra
\textbf{28}, 3865--3897 (2000)
%
\bibitem{jane83}
Jankins, M., Neumann, W. D.:
Lectures on Seifert Manifolds,
Brandeis Lecture Notes \textbf{2},
Brandeis University, Waltham, MA (1983).\\
\verb+http://www.math.columbia.edu/~neumann/preprints/+
%
\bibitem{lrs15}
Levine, A. S., Ruberman, D., Strle, S.:
Nonorientable surfaces in homology cobordisms.
Geom. Topol.
\textbf{19}, 439--494 (2015)
%
\bibitem{scot83}
Scott, P.:
The geometries of $3$-manifolds.
Bull. London Math. Soc.
\textbf{15}, 401--487 (1983)
%
\end{thebibliography}
\end{document}